\documentclass[preprint,10pt,3p]{elsarticle}
\usepackage{amsmath,amsthm,amsfonts,amssymb}
\usepackage{xcolor}
\usepackage{sectsty}
\usepackage{tabularx,booktabs}

\definecolor{astral}{RGB}{46,116,181}
\sectionfont{\color{astral}}
\newtheorem{theorem}{Theorem}[section]

\newtheorem{remark}{Remark}[section]

\usepackage{hyperref}
\usepackage{amsmath}
\usepackage{amsfonts}
\usepackage{amssymb,enumerate}
\usepackage{color}
\usepackage{xcolor}

\numberwithin{equation}{section}
    
\begin{document}
\begin{frontmatter}
\title{An augmented Lagrangian-based preconditioning technique for  a class of block three-by-three  linear systems}
\author{Fatemeh P. A. Beik$^{a,b}$ and Michele Benzi$^{c}$}

\address{$^{a}$Department of Mathematics, Vali-e-Asr University of Rafsanjan, P.O. Box 518, Rafsanjan, Iran\\
	$^{b}$Department of Mathematics, Simon Fraser University, Burnaby, BC, Canada\\
	$^{a,b}${\tt f.beik@vru.ac.ir;fatemeh$_{-}$panjeh$_{-}$ali$_{-}$beik@sfu.ca}\\
	$^{c}$Scuola Normale Superiore, Piazza dei Cavalieri, 7, 56126, Pisa, Italy\\
	$^{c}${\tt michele.benzi@sns.it }
}

\begin{abstract}
	
We propose an augmented Lagrangian-based preconditioner to accelerate the convergence of Krylov subspace methods applied to linear systems of equations
with a block  three-by-three  structure such as those arising from 
mixed finite element discretizations of the coupled Stokes-Darcy flow problem. We analyze the spectrum of the preconditioned matrix and we show how the  new
preconditioner can be efficiently applied. Numerical experiments  are reported to illustrate the effectiveness of the preconditioner in conjunction with flexible GMRES for solving 
linear systems of equations arising from a 3D test problem.
\end{abstract}


\end{frontmatter}

%%%%%%%%%%%%%%%%%%%%%%%%%%%%
%%%%%%%%%%%%%%%%%%%%%%%%
\section{Introduction} \label{sec1}
Consider the following $(n+m+p)\times (n+m+p)$ linear system of equations:
\begin{equation}\label{eq1}
	\mathcal{A}u = \left[ {\begin{array}{*{20}c}
			{A_{11} } & {A_{{12} } } & 0  \\
			{ A_{21 } } & {A_{22 } } & {B^T }  \\
			0 & B & 0  \\
	\end{array}} \right]\left[ {\begin{array}{*{20}c}
			{u_1 }  \\
			{u_2 }  \\
			{u_3}  \\
	\end{array}} \right] = \left[ {\begin{array}{*{20}c}
			{b_1 }  \\
			{b_2 }  \\
			{b_3}  \\
	\end{array}} \right] = b,
\end{equation}
where $A_{11}$ and $A_{22}$  are both symmetric positive definite (SPD), $A_{21}=-A_{12}^T$ and $B$ has full row rank.
In this paper, we are especially interested in the case where the above system corresponds to inf-sup stable mixed finite element discretizations of the coupled Stokes-Darcy flow problem; 
see \cite{Cai,Chid} for further details. Krylov subspace methods (such as GMRES) in conjunction with suitable preconditioners  are frequently the method of choice
for computing approximate solutions  of such linear systems of equations; see \cite{BB2022,Cai,Chid} and the references therein. 

%The numerical experiments reported in  \cite{BB2022} indicate that among all examined inexact variants of preconditioners, the augmented Lagrangian based preconditioner with IC-CG inner solvers leads to the fastest convergence speed of the FGMRES method in term of total solution times by a large margin comparing to the preconditioners proposed in \cite{Chid,Cui}.

In \cite{BB2022}, first, problem \eqref{eq1} is reformulated as the equivalent {\em augmented}  system $\bar{\mathcal{A}}u=\bar{b}$, where
\begin{equation}\label{eq183}
	\bar{\mathcal{A}}=\left[ {\begin{array}{*{20}{c}}
			{{A_{11}}}&{{A_{12}}}&0\\
			{{A_{21}}}&{{A_{22}} + \gamma {B^T}{Q^{ - 1}}B}&{{B^T}}\\
			0&B&0
	\end{array}} \right],
\end{equation}
and
$\bar{b}=(b_1; b_2+\gamma B^TQ^{-1}b_3; b_3)$, with $Q$ being an arbitrary SPD matrix and $\gamma > 0$  a user-defined
parameter. Then, the following preconditioner is proposed:
\begin{equation*}
	\mathcal{P}_{\gamma}=\left[ {\begin{array}{*{20}{c}}
			{{A_{11}}}&{{A_{12}}}&0\\
			0&{{A_{22}} + \gamma {B^T}{Q^{ - 1}}B}&{{B^T}}\\
			0&0&{ - \frac{1}{\gamma}Q}
	\end{array}} \right].
\end{equation*}
In practice, this preconditioner is applied inexactly by means of inner iterations and thus is used with Flexible GMRES (FGMRES) \cite{Saad}.
It was observed in \cite{BB2022} that FGMRES with the inexact augmented Lagrangian-based preconditioner ${\mathcal P}_\gamma$
exhibits faster convergence for larger values of $\gamma$.
However, for large $\gamma$ the total timings increase due to the fact that the condition number of the block ${{A_{22}} + \gamma {B^T}{Q^{ - 1}}B}$
goes up as $\gamma$ increases. Hence, the conjugate gradient (CG) method used for solving subsystems 
with coefficient matrix ${{A_{22}} + \gamma {B^T}{Q^{ - 1}}B}$ needs to be applied with a preconditioner. 
Explicitly forming ${{A_{22}} + \gamma {B^T}{Q^{ - 1}}B}$ to compute an incomplete Cholesky factorization leads to a considerably less sparse matrix and superlinear growth in
the fill-in in the incomplete factors, and thus to more expensive preconditioned CG (PCG) iterations; for further details see \cite[Table 2]{BB2022}.  

In this paper, we introduce a new class of preconditioners for $\bar{\mathcal{A}}$ given as follows:
\begin{equation}\label{new_pre}
\mathcal{P}_{\gamma,\alpha}= \left[{\begin{array}{*{20}{c}}
			{{A_{11}}}&{{A_{12}}}&0\\
			0&{{A_{22}} + \gamma {B^T}{Q^{ - 1}}B}&{(1 - {\gamma}{\alpha^{-1}}){B^T}}\\ 
			{0}&B&{ -{\alpha^{-1}}Q}
	\end{array}} \right]
\end{equation}
where  $\alpha$ and $\gamma$ are prescribed positive parameters such that $\alpha \ge \gamma$.  %and other blocks have the similar structure as before. 
As shown below, this form of preconditioning allows us to avoid the requirement  of forming the augmented block $A_{22} + \gamma {B^T}{Q^{ - 1}}B$,
which makes it possible to work with large values of $\gamma$. It is also highly effective in reducing the number of FGMRES iterations.

The rest of paper is organized as follows. In section \ref{sec2}, we derive some bounds for the eigenvalues
of the preconditioned matrix  $\bar{\mathcal{A}}\mathcal{P}_{\gamma,\alpha}^{-1}$. Numerical results are 
presented in section \ref{sec3} which illustrate the effectiveness of the proposed preconditioner in conjunction with FGMRES  in terms of both 
the number of iterations and the CPU time.  Section \ref{sec4} concludes the paper.\\

\noindent \textbf{Notations.} Given a square matrix $A$, the set of all eigenvalues (spectrum) of $A$ is denoted by $\sigma(A)$. When the spectrum of $A$ is real, we use $\lambda_{\min}(A)$ and $\lambda_{\max}(A)$ to respectively denote its minimum and maximum eigenvalues. When $A$ is symmetric positive (semi)definite, we write $A\succ 0$ ($A\succcurlyeq 0$). In addition, for two given matrices $A$ and $B$, the relation $A\succ B$ ($A\succcurlyeq B$) means $A-B\succ 0$ ($A-B\succcurlyeq 0$). Finally, for vectors $x$,  $y$ and $z$ of dimensions $n$, $m$ and $p$, $(x;y;z)$ will denote a column vector of dimension $n+m+p$. Throughout the paper, $I$ will denote the identity matrix (the size of
which will be clear from the context).

\section{Main results}\label{sec2}

In this section, we first obtain some bounds for the eigenvalues of $\mathcal{P}_{\gamma,\alpha}^{-1}\bar{\mathcal{A}}$ 
(which coincide with those of $\bar{\mathcal{A}}\mathcal{P}_{\gamma,\alpha}^{-1}$; recall that only right preconditioning 
is allowed with FGMRES ). We show that using large
values of $\alpha$ leads to a well clustered eigenvalue distribution for the preconditioned matrix. Then we explain how the preconditioner can be efficiently
implemented.  

\begin{theorem}\label{{thm1}}
Let $\bar{\mathcal{A}}$ and $\mathcal{P}_{\gamma,\alpha}$  be respectively defined by \eqref{eq183} and \eqref{new_pre}. The eigenvalues of  $\sigma(\mathcal{P}_{\gamma,\alpha}^{-1}\bar{\mathcal{A}})$ are all real and positive. More precisely, for an arbitrary $\lambda \in  \sigma(\mathcal{P}_{\gamma,\alpha}^{-1}\bar{\mathcal{A}}) $, we have
\[\frac{\xi^2 \alpha \lambda_{\min}(Q)}{\lambda_{\max}(Q)((\lambda_{\max}(A_{22})+\lambda_{\max}(A_{12}^TA_{11}^{-1}A_{12}))\lambda_{\min}(Q)+2\alpha\|B\|_2^2)} \le \lambda < 2+	\frac{\lambda_{\max}(A_{12}^TA_{11}^{ - 1}A_{12})}{\lambda_{\min}(A_{22})}
\]
with $
\xi = \mathop {\min }\limits_{} \left\{ {\left. \|By\|_2\,  \right| \, y \notin {\rm Ker}(B)}, \, y^*y=1 \right\}.
$
\end{theorem}

\begin{proof} For simplicity, we set $\bar{A}_{22}=A_{22}+\gamma B^TQ^{-1}B$. Let $\lambda$ and $(x;y;z)$ be an arbitrary eigenpair of 
$\mathcal{P}_{\gamma,\alpha}^{-1}\bar{\mathcal{A}}$. Therefore, we have
		\begin{subequations}\label{eig}
		\begin{align}
		A_{11}x+A_{12}y& = \lambda (A_{11}x+A_{12}y)\label{eig1}\\
	A_{21}x+\bar{A}_{22}y +B^Tz & = \lambda (\bar{A}_{22}y + (1-{\gamma}{\alpha}^{-1})B^Tz)\label{eig2}\\
	By &= \lambda  (By - {\alpha}^{-1} Qz)\label{eig3}
		\end{align}
	\end{subequations}
From Eq. \eqref{eig1}, we deduce that either $\lambda=1$ or $x=-A_{11}^{-1}A_{12}y$. In particular, one can verify that $\lambda=1$ belongs to the spectrum of  
$\mathcal{P}_{\gamma,\alpha}^{-1}\bar{\mathcal{A}}$ with a corresponding eigenvector 
of the form $(x;y;0)$ provided that $x$ and $y$ are not simultaneously zero and that $x\in \text{Ker} (A_{21})$.  

In the rest of the proof, we assume that $\lambda \ne 1$. Notice that if $y \in \text{Ker} (B)$, then Eq. \eqref{eig3} implies that $z$ is the zero vector. Hence, from \eqref{eig2}, one can conclude that
\begin{equation}\label{eq2.2}
	\lambda=1+ \frac{{y^* A_{12}^TA_{11}^{- 1}{A_{12}}y}}{y^*\bar{A}_{22}y}=1+ \frac{{y^* A_{12}^TA_{11}^{- 1}{A_{12}}y}}{y^*{A}_{22}y}.
\end{equation}

In the sequel, without loss of generality, we may assume that $y\notin \text{Ker} (B)$ and $y^*y=1$. Since $\lambda \ne 1$, from \eqref{eig3}, we further obtain
\[
z=\alpha \left(\frac{\lambda-1}{\lambda}\right) Q^{-1}By.
\]
Substituting $z$ from the above relation and $x=-A_{11}^{-1}A_{12}y$ in \eqref{eig2}, we obtain
\[
A_{12}^TA_{11}^{-1}A_{12}y+(1-\lambda) \bar{A}_{22}y+\left(1-\lambda (1-{\gamma}{\alpha}^{-1})\right)\left(\frac{\lambda-1}{\lambda}\right) \left(\alpha B^TQ^{-1}By\right)=0\,,
\]
having in mind that $A_{21}=-A_{12}^T$. For ease of notation we set 
\[
p:= y^*A_{12}^TA_{11}^{-1}A_{12}y, \quad q:=y^*\bar{A}_{22}y \quad \text{and} \quad t:= \alpha y^*B^TQ^{-1}By.
\] 
Left-multiplying both sides of the preceding relation by $\lambda y^*$, we reach to the following quadratic equation:
\[
\left(1+  \frac{t}{q}(1-{\gamma}{\alpha}^{-1})\right)\lambda^2-\left(1+  \frac{t}{q}(1-{\gamma}{\alpha}^{-1})+\frac{t}{q}+\frac{p}{q} \right)\lambda +\frac{t}{q}=0
\]
or, equivalently, 
\begin{equation}\label{eqq}
	\lambda^2-b\lambda +c=0
\end{equation}
where
\begin{equation}\label{coef}
	b:=1+ \frac{p+t}{q+(1-\frac{\gamma}{\alpha})t} \quad \text{and} \quad c:=\frac{t}{q+(1-\frac{\gamma}{\alpha})t}.
\end{equation}

Notice that $b \ge 1+c$. As a result, it is immediate to see that the roots of \eqref{eqq} are real and given by
\[
{\lambda _1} = \frac{b  - \sqrt {b ^2 - 4c}}{2}
\quad
\text{and}
\quad
{\lambda _2} = \frac{b  + \sqrt {b ^2 - 4c }}{2}.
\]
It is not difficult to see that 
\begin{eqnarray}
	\nonumber	{\lambda _1} &= & \frac{{2c}}{{b  + \sqrt {{b ^2} - 4c } }} \\
   \nonumber           	& \ge & \frac{c }{b} ={\frac{\alpha y^*B^TQ^{-1}By}{y^*A_{22}y+p+2\alpha y^*B^TQ^{-1}By}} \\
              	\nonumber \\
\nonumber		& \ge &\frac{\alpha \lambda_{\min}(Q^{-1})\|By\|_2^2}{(\lambda_{\max}(A_{22})+\lambda_{\max}(A_{12}^TA_{11}^{-1}A_{12}))+2\alpha \|B\|_2^2\lambda_{\max}(Q^{-1})} \\
\nonumber \\
\nonumber  &\ge & \frac{\xi^2 \alpha \lambda_{\min}(Q)}{\lambda_{\max}(Q)((\lambda_{\max}(A_{22})+\lambda_{\max}(A_{12}^TA_{11}^{-1}A_{12}))\lambda_{\min}(Q)+2\alpha\|B\|_2^2)}
\end{eqnarray}
where
$
\xi = \mathop {\min }\limits_{} \left\{ {\left. \|By\|_2\, \right| \, y \notin {\rm Ker}(B)}, \, y^*y =1 \right\}.
$
Also, it can be observed that 
\begin{eqnarray}
\nonumber \lambda_2 &\le& b  \\
&= &1+\frac{p+\alpha y^*B^TQ^{-1}By}{y^*A_{22}y+\alpha y^*B^TQ^{-1}By} \label{eq2.5}\\
\nonumber \\
\nonumber &< & 2+  \frac{\lambda_{\max}(A_{12}^TA_{11}^{ - 1}A_{12})}{\lambda_{\min}(A_{22})}.
\end{eqnarray}
\end{proof}

\begin{remark}\label{rm1}
	{\rm
	As pointed out in \cite{BB2022},  for problems of small
	or moderate size, it can be numerically checked that  the condition $A_{22} \succ A_{12}^T A_{11}^{-1} A_{12}$ is satisfied for linear systems of the form (\ref{eq1}) arising from the finite element discretization of coupled Stokes-Darcy flow. Under this condition, for any $\lambda \in \sigma(\mathcal{P}_{\gamma,\alpha}^{-1}\bar{\mathcal{A}})$ and $\alpha >0$, from Eqs. \eqref{eq2.2} and \eqref{eq2.5}, we can deduce that $0<\lambda \le 1 + \tau$ for some $\tau <1$. 	
	The previous theorem implies that for fixed  $n$, $m$ and $p$, as  $\alpha \to \infty$, except for the possible eigenvalues given by \eqref{eq2.2}, any eigenvalue $\lambda$ 
	(with $\lambda \ne 1$) of the preconditioned matrix tends to 1 as $\alpha \to \infty$ since in this case the coefficients $b$ and $c$ in Eq. \eqref{coef}  go to $2$ and $1$, respectively. Consequently, the left-hand side of the quadratic equation \eqref{eqq} tends to $(\lambda-1)^2$ as $\alpha \to \infty$.  This behavior 
	appears to be confirmed in Fig.~\ref{fig:1}. While eigenvalues alone do not fully describe the convergence of Krylov subspace methods for nonsymmetric matrices, a well
	clustered spectrum away from zero is often associated with rapid convergence.

	\begin{figure}[t]
		\includegraphics[width=1\textwidth]{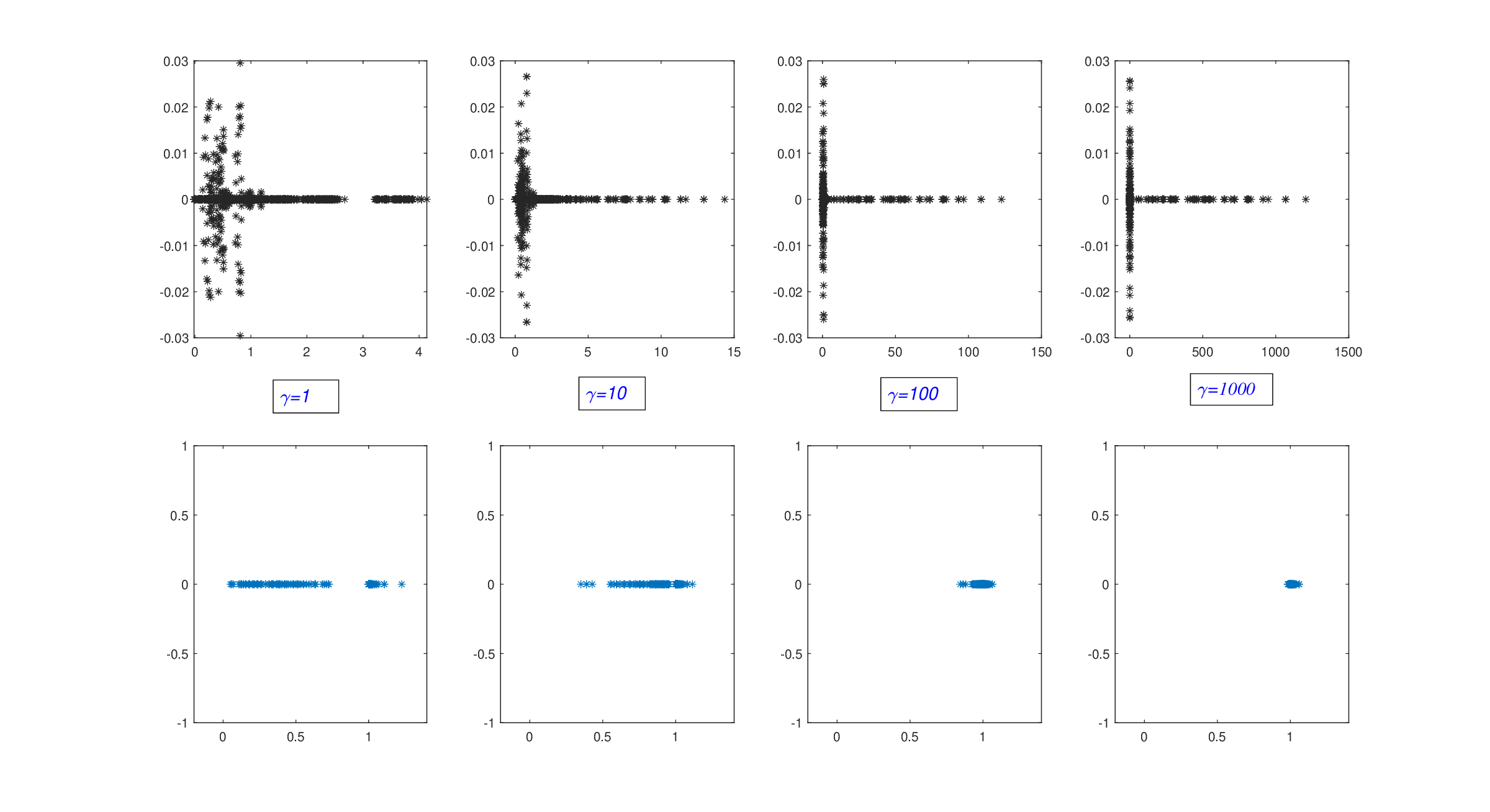}
		\caption{Eigenvalue distributions of $\bar {\mathcal A}$ (top)  versus that of the preconditioned matrix {${\mathcal P}_{\gamma,\alpha}^{-1}\bar {\mathcal A}$}
			(bottom) for different values of $\gamma$, with $Q=\text{diag}(M_p)$ and $\alpha=2\gamma$ for a 3D coupled Stokes-Darcy problem with 1695 degrees of freedom.}
		\label{fig:1}	
	\end{figure}
}
\end{remark}

We end this section by some brief comments on the implementation 
of the proposed preconditioner inside FGMRES. While we do not claim that this implementation is
optimal, it gives good results in practice and is competitive with more sophisticated multi-level solvers while having
much lower lower set-up costs.
To apply the preconditioner, in each inner iteration we need to 
solve linear systems of the form
\[
\mathcal{P}_{\gamma,\alpha}(w_1;w_2;w_3)=(r_1;r_2;r_3).
\] 
To this end, we use the following block factorization 
\[\mathcal{P}_{\gamma,\alpha}=\left[ {\begin{array}{*{20}{c}}
		I&0&0\\
		0&I&{\gamma {B^T}Q^{-1}}\\
		0&0&I
\end{array}} \right]\left[ {\begin{array}{*{20}{c}}
		{{A_{11}}}&{{A_{12}}}&0\\
		0&{{A_{22}}}&{{B^T}}\\
		0&B&{ - {\alpha ^{ - 1}}Q}
\end{array}} \right]\] 
which allows one to work with larger values of $\gamma$. In fact, we do not have to solve 
linear systems with coefficient matrix  $A_{22} + \gamma {B^T}{Q^{ - 1}}B$. To apply the preconditioner, we need to solve subsystems of the form
\begin{equation}\label{inside}
	\left[ {\begin{array}{*{20}{c}}
			{{A_{22}}}&{{B^T}}\\
			B&{ - {\alpha ^{ - 1}}Q}
	\end{array}} \right]\left[ {\begin{array}{*{20}{c}}
			{{w_2}}\\
			{{w_3}}
	\end{array}} \right] = \left[ {\begin{array}{*{20}{c}}
			{{r_2-\gamma B^TQ^{-1}r_3}}\\
			{{r_2}}
	\end{array}} \right].
\end{equation}
(note that these systems are of stabilized Stokes type, see \cite{Elman}).
To do so, we use GMRES (with a loose stopping residual tolerance $0.1$) in conjunction with the following block triangular preconditioner:
\[P = \left[ {\begin{array}{*{20}{c}}
		\hat{A}_{22}&0\\
		B&{ - \hat S}
\end{array}} \right]\]
in which  $\hat A_{22}$ and
$\hat S$ are approximations of 
$A_{22}$ and $S=\alpha^{-1}Q+M_p$ obtained via incomplete Cholesky factorizations constructed by MATLAB
function
``\verb|ichol|(., \verb|opts|)" and MATLAB backslash operator ``$\backslash$", 
with {\verb|opts.type| ='\verb|ict|'} and {\verb|opts.droptol| =$\epsilon_i$} where $\epsilon_i$ is respectively equal to  $10^{-3}$ and $10^{-2}$ for $i=1,2$ where $M_p$ denotes the mass matrix coming  from the Stokes pressure space. We further comment that the matrix $Q$ in our numerical implementation is $\text{diag}(M_p)$ and $\alpha=2\gamma$, which means that effectively only the parameter $\gamma$ has to be set by the user. 
As seen, we also need to solve the linear systems with the coefficient matrix $A_{11}$ which is solved by PCG with an incomplete 
Cholesky factorization, as in \cite{BB2022}. The inner PCG iteration was terminated when the relative residual norm was below $10^{-1}$ or when the maximum number of $5$  iterations was reached. These parameter choices, while probably not optimal, were find to be a good compromise in terms of simplicity of implementation, low set-up costs and 
good preconditioner effectiveness and robustness.

\section{Numerical experiments}\label{sec3}

In this section we report on the performance of inexact variants of the proposed block preconditioner 
using a test problem taken from \cite[Subsection 5.3]{Chid}, which corresponds to a 3D coupled flow problem with large jumps in the permeability
in the porous flow region. 
All computations were carried out on a computer with an Intel Core i7-10750H CPU @ 2.60GHz processor and 16.0GB RAM using MATLAB.R2020b.

In Tables 1 and 2 we report the total required number of outer FGMRES iterations and elapsed CPU time (in seconds) under ``Iter" and ``CPU", respectively.
The total number of inner GMRES (PCG) iterations to solve subsystems \eqref{inside} (respectively, with coefficient matrix $A_{11}$)
is reported under Iter$_{{in}}$ (Iter$_{{pcg}}$). No restart is used for either FGMRES or GMRES iterations.

The initial guess is taken to be the zero vector and the iterations are stopped as soon as
$$\|\bar{\mathcal{A}}u_k-\bar{b}\|_2{\le 10^{-7}} \|\bar{b}\|_2$$
where $u_k$ is the computed $k$-th approximate solution. 
In the  tables, we also include the relative error
\[
\mathrm{Err}:=\frac{{\|u_k-u^\ast\|_2}}{{\|u^\ast\|_2}},
\]
where $u^\ast$ and $u_k$ are, respectively, the exact solution and  its approximation obtained in the $k$-th iterate.
In addition, we have used right-hand sides corresponding to random solution vectors and averaged results over
10 test runs, rounding the iteration counts to the nearest integer. 

The results show that in all cases, the outer FGMRES iteration displays mesh-independent convergence. The number of inner
iterations, on the other hand, increases as the mesh is refined, especially for the largest test problem, resulting in less than perfect 
scalability of the solver. Nevertheless, the total solution time is about 40\% less, for the largest problem, than the best results
reported in \cite{BB2022}, obtained using the multi-level preconditioner ARMS in the solution of the inner subproblems,
showing the considerable advantage of the new augmented Lagrangian preconditioner over the original
version.

\begin{table*}[t]
	\caption{Results for FGMRES in conjunction with preconditioner $\mathcal{P}_{\gamma,2\gamma}$.}\label{tab1}
	\resizebox{\textwidth}{!}{	\centering
		\begin{tabular}{@{}lcccccccccccc@{}}
			\toprule
			&& \multicolumn{5}{c}{{$\gamma =1$}} &&  \multicolumn{5}{c}{{$\gamma =10$}} \\
			\cmidrule{3-7} \cmidrule{9-13}
			size && \multicolumn{2}{c}{{FGMRES}} & \phantom{abc} &\multicolumn{2}{c}{Inner iterations} && \multicolumn{2}{c}{{FGMRES}} & \phantom{abc} &\multicolumn{2}{c}{Inner iterations} \\
			\cmidrule{3-4}\cmidrule{6-7} \cmidrule{9-10}\cmidrule{12-13}
			&& Iter (CPU) &Err&& Iter$_{{in}}$ & Iter$_{{pcg}}$  && Iter (CPU) & Err && Iter$_{in}$ & Iter$_{{pcg}}$\\
			\midrule
			1695  && 22(0.05) & 6.2804e-05  && 128  & 56 && 12(0.02) & 7.3326e-06 &&75 & 33 \\
			10809 && 20(0.52) & 1.0262e-04  && 130  & 49 && 12(0.34) & 4.1502e-06 &&88 & 32  \\
			76653 && 19(4.66) & 1.9242e-04  && 148  & 58 && 11(3.02) & 2.2011e-05 &&99 & 35   \\
			576213&& 19(64.9) & 3.8861e-04  && 248  & 85 && 11(38.1) & 5.5276e-05 &&146 & 45   \\
			\bottomrule	
	\end{tabular}}
\end{table*} 
		
\begin{table*}[t]
	\caption{Results for FGMRES  in conjunction with preconditioner $\mathcal{P}_{\gamma,2\gamma}$.}\label{tab2}
	\resizebox{\textwidth}{!}{	\centering
		\begin{tabular}{@{}lcccccccccccc@{}}
			\toprule
			&& \multicolumn{5}{c}{{$\gamma=100$}} &&  \multicolumn{5}{c}{{$\gamma=1000$}} \\
			\cmidrule{3-7} \cmidrule{9-13}
			size && \multicolumn{2}{c}{{FGMRES}} & \phantom{abc} &\multicolumn{2}{c}{Inner iterations} && \multicolumn{2}{c}{{FGMRES}} & \phantom{abc} &\multicolumn{2}{c}{Inner iterations} \\
			\cmidrule{3-4}\cmidrule{6-7} \cmidrule{9-10}\cmidrule{12-13}
			&& Iter (CPU) &Err&& Iter$_{{in}}$ & Iter$_{{pcg}}$  && Iter (CPU) & Err && Iter$_{in}$ & Iter$_{{pcg}}$\\
			\midrule
			1695  && 11(0.02) & 2.6911e-06  && 63  & 29 && 11(0.02) & 2.3876e-06 &&62 & 28 \\
			10809 && 11(0.29) & 1.1382e-06  && 71  & 26 && 11(0.28) & 1.5137e-06 &&69 & 25  \\
			76653 && 11(2.70) & 5.1225e-06  && 86  & 30 && 11(2.66) & 2.9047e-06 &&86 & 29   \\
			576213&& 11(36.4) & 1.4794e-05 && 133  & 38 && 11(35.6) & 1.3158e-05 &&130 & 37   \\
			\bottomrule	
	\end{tabular}}
\end{table*}

\section{Conclusions}\label{sec4}
In this paper we have introduced a new augmented Lagrangian-based preconditioner for block three-by-three linear systems, with a focus
on the linear systems arising from finite element discretizations of the coupled Darcy-Stokes flow problem. The main advantage of this
preconditioner is that it avoids the need to explicitly form the augmented block $A_{22} + \gamma B^TQ^{-1}B$. Theoretical analysis shows a
strong clustering of the spectrum of the (exactly) preconditioned matrix, and numerical experiments on a challenging 3D model problem show that the corresponding inexact
preconditioner can result in much faster convergence than previous versions of the augmented Lagrangian-based preconditioner. 

Future work will focus on replacing the incomplete Cholesky inner preconditioners with multilevel preconditioners, in order to obtain better
scalability of the solver.\\

%\vspace{0.2in}

%\noindent {\bf Data Availability Statement.}

%\noindent {\bf {Conflicts of interest.}} The authors declare no conflict of interest.

\noindent {\bf{ Acknowledgement.} } The authors would like to thank Scott Ladenheim for providing the test problem. The work of M.~Benzi was supported in part by ``Fondi per la Ricerca di Base" of the Scuola Normale Superiore. Thanks also to two anonymous referees for their helpful suggestions.

\bibliographystyle{abbrv}
%\bibliography{Reference}

%\begin{thebibliography}{10}

\end{document}